\documentclass[a4paper,12pt,final]{amsart}
\usepackage{times,a4wide,mathrsfs,amssymb,amsmath,amsthm,enumerate,xypic,tikzsymbols,dsfont}

\newcommand{\C}{\mathbb{C}}
\newcommand{\ZZ}{\mathbb{Z}}

\newcommand{\QQ}{\mathbb{Q}}

\newcommand{\PP}{\mathbb{P}}

\newcommand{\OO}{\mathcal O}

\newcommand{\YY}{\mathcal Y}
\newcommand{\UU}{\mathcal U}

\newcommand{\WW}{\mathcal W}

\newcommand{\MM}{\mathcal M}

\newcommand{\codim}{\hbox{codim}}

\newcommand{\wt}{\widetilde}
\newcommand{\rom}{\romannumeral}

\DeclareMathOperator{\rank}{rank}

\DeclareMathOperator{\Gr}{Gr}

\newtheorem{theorem}{Theorem}[section]

\newtheorem{corollary}[theorem]{Corollary}
\newtheorem{proposition}[theorem]{Proposition}
\newtheorem{conjecture}[theorem]{Conjecture}
\newtheorem{convention}{Conventions}

\newtheorem{nonumbering}{Theorem}

\newtheorem{nonumberingc}{Corollary}

\theoremstyle{definition}
\newtheorem{remark}[theorem]{Remark}

\newtheorem{nonumberingt}{Acknowledgments}

\begin{document}

\author[Robert Laterveer]
{Robert Laterveer}

\address{Institut de Recherche Math\'ematique Avanc\'ee,
CNRS -- Universit\'e 
de Strasbourg,\
7 Rue Ren\'e Des\-car\-tes, 67084 Strasbourg CEDEX,
FRANCE.}
\email{robert.laterveer@math.unistra.fr}

\title{On the Chow groups of hypersurfaces in symplectic Grassmannians}

\begin{abstract} Let $Y$ be a Pl\"ucker hypersurface in a symplectic Grassmannian $I_1 \Gr(3,n)$ or a bisymplectic Grassmannian $I_2 \Gr(3,n)$. We show that many Chow groups of $Y$ inject into cohomology.
\end{abstract}


\maketitle

\section{Introduction}

\noindent
Given a smooth projective variety $Y$ over $\C$, let $A_i(Y):=CH_i(Y)_{\QQ}$ denote the Chow groups of $Y$ (i.e. the groups of $i$-dimensional algebraic cycles on $Y$ with $\QQ$-coefficients, modulo rational equivalence). Let $A_i^{hom}(Y)\subset A_i(Y)$ denote the subgroup of homologically trivial cycles.

The famous Bloch--Beilinson conjectures \cite{Jan}, \cite{Vo} predict that the Hodge level of the cohomology of $Y$ should have an influence on the size of the Chow groups of $Y$. For surfaces, this is the notorious Bloch conjecture, which is still an open problem. For hypersurfaces in projective space, the precise prediction is as follows:

\begin{conjecture}\label{conj0} Let $Y\subset\PP^{n}$ be a smooth hypersurface of degree $d$. Then
  \[ A_i^{hom}(Y)=0\ \ \ \forall\ i\le {n\over d} -1\ .\]
\end{conjecture}

Conjecture \ref{conj0} is still open; partial results have been obtained in \cite{Lew}, \cite{V96}, \cite{Ot}, \cite{ELV}, \cite{HI}.

%
    
 In \cite{hyper}, I considered a version of Conjecture \ref{conj0} for Pl\"ucker hyperplane sections of Grassmannians. 
 In this note, we look at the case of Pl\"ucker hyperplane sections of {\em symplectic Grassmannians\/}. Recall that inside the Grassmannian $\Gr(3,n)$ (of $3$-dimensional subspaces of an $n$-dimensional vector space), the symplectic Grassmannian $I_1 \Gr(3,n)\subset\Gr(3,n)$  parametrizes subspaces that are isotropic with respect to some fixed skew-symmetric 2-form.
  The precise prediction (cf. subsection \ref{ss:mot} below) is as follows:
  
 \begin{conjecture}\label{conj} Let
  \[ Y:= I_1\Gr(3,n)\cap H \ \ \ \subset\ \PP^{{n\choose 3}-1} \]
  be a smooth hyperplane section (with respect to the Pl\"ucker embedding). Then
    \[ A_i^{hom}(Y)=0\ \ \ \forall\  i\le n-4\ .\]
    \end{conjecture}

    The main result of this note is a partial verification of Conjecture \ref{conj}:
    
    \begin{nonumbering}[=Theorem \ref{main}] Let  
      \[ Y:= I_1\Gr(3,n)\cap H \ \ \ \subset\ \PP^{{n\choose 3}-1} \]
  be a smooth hyperplane section (with respect to the Pl\"ucker embedding).   
  Then
      \[ A_i^{hom}(Y)=0\ \ \ \forall\ i\le n-5\ .\]
      Moreover, in case $n\le 10$ or $n=12$ we have
      \[  A_{i}^{hom}(Y)=0\ \ \ \forall\ i\le n-4\ .\]
        \end{nonumbering}
    
To prove Theorem \ref{main}, we rely on the recent notion of {\em projections\/} among (symplectic) Grassmannians \cite{BFM}. Combined with the Chow-theoretic Cayley trick \cite{Ji}, this reduces Theorem \ref{main} to understanding the Chow groups of a hyperplane section in an ordinary Grassmannian $\Gr(3,n+1)$. This last problem was handled in \cite{hyper}.

As a consequence of Theorem \ref{main}, some instances of the generalized Hodge conjecture are verified:

\begin{nonumberingc}[=Corollary \ref{ghc}] Let $Y$ be as in Theorem \ref{main}, and assume $n\le 10$ or $n=12$. Then $H^{\dim Y}(Y,\QQ)$ is supported on a subvariety of codimension $n-3$.
\end{nonumberingc}

Other consequences are as follows:

\begin{nonumberingc}[=Corollary \ref{cor2}] Let  
\[ Y:=  I_1 \Gr(3,n)\cap H \ \ \ \subset\ \PP^{{n\choose 3}-1}\]
  be a smooth hyperplane section (with respect to the Pl\"ucker embedding). 
  
 \noindent
 (\rom1) If $n\le 8$, then $Y$ has finite-dimensional motive (in the sense of \cite{Kim}).
 
 \noindent
 (\rom2) If $n\le 9$, then $Y$ has trivial Griffiths groups (and so Voevodsky's smash conjecture is true for $Y$).
 
 \noindent
 (\rom3) If $n\le 10$, the Hodge conjecture is true for $Y$.
  \end{nonumberingc}

Applying the same method, we can also say something about hypersurfaces in {\em bisymplectic Grassmannians\/}. (Recall that the bisymplectic Grassmannian $ I_2 \Gr(3,n)\subset\Gr(3,n)$ is the locus of 3-spaces that are isotropic with respect to two fixed generic skew-forms.)

   \begin{nonumbering}[=Theorem \ref{main2}] Assume $n$ is even, and let
      \[ Y:= I_2 \Gr(3,n)\cap H \ \ \ \subset\ \PP^{{n\choose 3}-1} \]
  be a smooth hyperplane section (with respect to the Pl\"ucker embedding).   
  Then
      \[ A_i^{hom}(Y)=0\ \ \ \forall\ i\le n-7\ .\]
      Moreover, in case $n\le 10$ we have
      \[  A_{i}^{hom}(Y)=0\ \ \ \forall\ i\le n-6\ .\]
        \end{nonumbering}

The hyperplane sections $I_1 \Gr(3,9)\cap H$ and $I_2 \Gr(3,8)\cap H$ are of particular interest: they are Fano varieties of K3 type, and they are related to hyperplane sections $\Gr(3,10)\cap H$ and hence to Debarre--Voisin hyperk\"ahler fourfolds, cf. \cite[Section 4]{BFM}.

 \vskip0.8cm

\begin{convention} In this note, the word {\sl variety\/} will refer to a reduced irreducible scheme of finite type over $\C$. A {\sl subvariety\/} is a (possibly reducible) reduced subscheme which is equidimensional. 

{\bf All Chow groups will be with rational coefficients}: we denote by $A_j(Y):=CH_j(Y)_{\QQ} $ the Chow group of $j$-dimensional cycles on $Y$ with $\QQ$-coefficients; for $Y$ smooth of dimension $n$ the notations $A_j(Y)$ and $A^{n-j}(Y)$ are used interchangeably. 
The notations $A^j_{hom}(Y)$ and $A^j_{AJ}(X)$ will be used to indicate the subgroup of homologically trivial (resp. Abel--Jacobi trivial) cycles.

The contravariant category of Chow motives (i.e., pure motives with respect to rational equivalence as in \cite{Sc}, \cite{MNP}) will be denoted 
$\MM_{\rm rat}$.
\end{convention}

 \vskip1.0cm

\section{Preliminaries}

 \subsection{Cayley's trick and Chow groups}

\begin{theorem}[Jiang \cite{Ji}]\label{ji} Let $ E\to U$ be a vector bundle of rank $r\ge 2$ over a projective variety $U$, and let $S:=s^{-1}(0)\subset U$ be the zero locus of a regular section $s\in H^0(U,E)$ such that $S$ is smooth of dimension $\dim U-\rank E$. Let $X:=w^{-1}(0)\subset \PP(E)$ be the zero locus of the regular section $w\in H^0(\PP(E),\OO_{\PP(E)}(1))$ that corresponds to $s$ under the natural isomorphism $H^0(U,E)\cong H^0(\PP(E),\OO_{\PP(E)}(1))$. 
There is an isomorphism of integral Chow motives
  \[  h(X)\cong h(S)(1-r)\oplus \bigoplus_{i=0}^{r-2} h(U)(-i)\ \ \ \hbox{in}\ \MM^{\ZZ}_{\rm rat}\ .\]
    \end{theorem}
    
    \begin{proof} This is \cite[Theorem 3.1]{Ji}. Both the isomorphism and its inverse are explicitly described.
     \end{proof}

\begin{remark} In the set-up of Theorem \ref{ji}, a cohomological relation between $X$, $S$ and $U$ was established in \cite[Prop. 4.3]{Ko} (cf. also \cite[section 3.7]{IM0}, as well as \cite[Proposition 46]{BFM} for a generalization). A relation on the level of derived categories was established in \cite[Theorem 2.10]{Or} (cf. also \cite[Theorem 2.4]{KKLL} and \cite[Proposition 47]{BFM}).
\end{remark}

\subsection{Linear sections of $\Gr(2,n)$}

\begin{proposition}\label{2n} Let
   \[ Y:= \Gr(2,n)\cap H_1\cap \cdots\cap H_s \ \ \ \subset\ \PP^{{n\choose 2}-1} \]
  be a smooth dimensionally transverse intersection with $s$ hyperplanes (with respect to the Pl\"ucker embedding).   
  Assume $s\le 2$. Then
      \[ A_i^{hom}(Y)=0\ \ \ \forall\ i\ .\]
\end{proposition}

\begin{proof} This uses a geometric construction that can be found in \cite{Don}.

Let $P\subset\PP(V_n)$ be a fixed hyperplane, and consider (as in \cite[Section 2.3]{Don}) the rational map
  \[  \Gr(2,V_n)\ \dashrightarrow\ P \]
  sending a line in $\PP(V_n)$ to its intersection with $P$. This map is resolved by blowing up a subvariety $\sigma_{11}(P)\cong \Gr(2,n-1)$, resulting in a morphism
   \[ \Gamma\colon\ \ \wt{\Gr}\ \to\ P \]
   (where $\wt{\Gr}\to \Gr(2,V_n)$ denotes the blow-up with center $\sigma_{11}(P)$).

Let $\wt{Y}\to Y$ be the blow-up of $Y$ with center $\sigma_{11}(P)\cap Y$, and let us consider the morphism
  \[ \Gamma_Y\colon\ \ \wt{Y}\ \to\ P \ ,\]
  obtained by restricting $\Gamma$.
  
  In case $s=1$ and $P$ is generic with respect to $Y$, the morphism $\Gamma_Y$ is a $\PP^{n-3}$-fibration over $P$. It follows that $\wt{Y}$, and hence $Y$, has trivial Chow groups.
  
  In case $s=2$, and $P$ chosen generically with respect to $Y$, the morphism $\Gamma_Y$ is generically a $\PP^{n-4}$-fibration over $P$, and there are finitely many points in $P$ where the fiber is $\PP^{n-3}$. Applying Theorem \ref{ji}, this implies that $\wt{Y}$, and hence also $Y$, has trivial Chow groups.
\end{proof}

  \subsection{Hyperplane sections of $\Gr(3,n)$}
  
  \begin{theorem}\label{3n} Let  
      \[ Y:= \Gr(3,n)\cap H \ \ \ \subset\ \PP^{{n\choose 3}-1} \]
  be a smooth hyperplane section (with respect to the Pl\"ucker embedding).   
  Then
      \[ A_i^{hom}(Y)=0\ \ \ \forall\ i\le n-3\ .\]
      Moreover, in case $n\le 11$ or $n=13$ we have
      \[  A_{i}^{hom}(Y)=0\ \ \ \forall\ i\le n-2\ .\]
    \end{theorem}
    
    \begin{proof} This is \cite[Theorems 3.1 and 3.2]{hyper}, which uses the notion of {\em jumps\/} between Grassmannians as developed in \cite{BFM}.
     \end{proof}

   \section{Main results}
   
   \subsection{Projections} As in \cite{BFM}, let $I_r \Gr(k,n)\subset \Gr(k,n)$ parametrize linear subspaces that are isotropic with respect to $r$ fixed generic skew-forms. One has
     \[ \dim I_r \Gr(k,n)=k(n-k)-r{k\choose 2}\ .\]
  For example, $I_r \Gr(2,n)$ is just the intersection of $\Gr(2,n)$ with $r$ Pl\"ucker hyperplanes. The case $I_2 \Gr(k,n)$ is studied in detail in \cite{Bene}.
    
  To relate hyperplane sections of different symplectic Grassmannians, Bernardara--Fatighenti--Manivel \cite{BFM} have developed a theory of {\em projections\/}.
   The starting point is a rational map
     \[ \pi\colon \Gr(k,n+1)\ \dashrightarrow\ \Gr(k,n) \ ,\]
     determined by the choice of a line in the $n+1$-dimensional vector space. If $Y^\prime$ is a hyperplane section of $I_r \Gr(k,n+1)$, one can restrict $\pi$ to $Y^\prime$. A detailed analysis
     of the case $k=3$ yields the following:
   
   \begin{theorem}[\cite{BFM}]\label{proj} Assume $n$ is even and $r\le 1$, or $n$ is odd and $r=0$. Let 
     \[Y^\prime:= I_r \Gr(3,n+1)\cap H\] 
     be a smooth hyperplane section. There exists
   a commutative diagram
     \[ \begin{array}[c]{ccccccc}
                    E & \hookrightarrow& \wt{Y}^\prime& \hookleftarrow & F &&\\
                    &&&&&&\\
                    \downarrow&&\ \ \downarrow{\scriptstyle\sigma}& \  \searrow{\scriptstyle \tau}&& \  \searrow{\scriptstyle q}   &\\
                    &&&&&&\\
                    Z^\prime&\hookrightarrow&Y^\prime&\stackrel{\pi^\prime}{\dashrightarrow}&I_r \Gr(3,n) &\hookleftarrow& Y\\
                    \end{array}\]
  where $Y:= I_{r+1} \Gr(3,n)\cap H$ is a smooth hyperplane section. The morphism $\sigma$ is the blow-up with center $Z^\prime\cong I_{r+1} \Gr(2,n)$. The morphism $q$ is a
   $\PP^3$-fibration, while $\tau$ is a $\PP^2$-fibration over the complement of $Y$.
    \end{theorem}
        
    \begin{proof} This is contained in \cite[Section 3.2]{BFM} (NB: note that our $n$ is $n-1$ in loc. cit.). As explained in loc. cit., the assumptions on $n$ and $r$ guarantee that the target $I_r \Gr(3,n)$ and the hyperplane section $Y$ are generic and hence smooth.
     \end{proof}

 \subsection{Motivating the conjecture}
 \label{ss:mot}
 
 As a consequence of Theorem \ref{proj}, one can compute the Hodge level of hyperplane sections $Y$ of symplectic Grassmannians: surprisingly, it turns out that (at least for $n>8$) $Y$ is of ``Calabi--Yau type'': 
  
 \begin{theorem}[\cite{BFM}]\label{coho} Let 
   \[ Y:= I_1\Gr(3,n)\cap H \ \ \ \subset\ \PP^{{n\choose k}-1} \]
  be a smooth hyperplane section (with respect to the Pl\"ucker embedding). 
  Assume $n>8$. Then $Y$ has Hodge coniveau $n-3$. More precisely, the Hodge numbers verify
    \[   h^{p,\dim Y-p}(Y) =\begin{cases}   1 & \hbox{if}\ p=n-3\ ,\\
                                                                 0 & \hbox{if}\ p<n-3\ .\\
                                                                 \end{cases}\]
                   \end{theorem}
                   
                   \begin{proof} This is implicit in \cite{BFM}, as we now explain. Let the set-up be as in Theorem \ref{proj}. Then
                    \cite[Proposition 6]{BFM} relates $Y$ and $Y^\prime$ on the level of cohomology: one has an isomorphism of Hodge structures
           \begin{equation}\label{hs}         \begin{split}    H^{j-6}(Y,\QQ)(-3) \oplus \bigoplus_{i=0}^2 H^{j-2i}(I_r\Gr(3,n),\QQ)(-i) \ \xrightarrow{\cong}&\  H^j(Y^\prime,\QQ) \oplus\\ 
           &\bigoplus_{i=1}^{c-1}  H^{j-2i}(I_{r+1}\Gr(2,n),\QQ)(-i)\\
                \end{split}\end{equation}
             (where $c$ denotes the codimension of $Z^\prime$ in $Y^\prime$).                    
           Setting $r=0$ and combining with \cite[Theorem 3]{BFM} (which gives the Hodge numbers of $Y^\prime$), plus the fact that $\Gr(3,n)$ and $I_1 \Gr(2,n)$ have algebraic cohomology, this gives the required Hodge numbers of $Y$.
                   \end{proof}
                   
   Theorem \ref{coho} motivates Conjecture \ref{conj}. Indeed,               
   the {\em generalized Bloch conjecture\/} \cite[Conjecture 1.10]{Vo} predicts that any variety $Y$ with Hodge coniveau $\ge c$ has
    \[ A_i^{hom}(Y)=0\ \ \ \forall\ i<c\ .\]                                                             
  Note that at least for $n>8$, the bound of Conjecture \ref{conj} is optimal: assuming $A_i^{hom}(Y)=0$ for $j\le n-3$ and applying the Bloch--Srinivas argument \cite{BS}, one would get the vanishing $h^{n-3,\dim Y -n+3}(Y)=0$, contradicting Theorem \ref{coho}.                                                            
 
   \subsection{A relation of motives}
     
   The cohomological relation \eqref{hs} between $Y$ and $Y^\prime$ also exists as (and actually is implied by) a relation on the level of the Grothendieck ring of varieties \cite[Proposition 4]{BFM}, and on the level of derived categories \cite[Proposition 5]{BFM}. To complete the picture, we now lift the relation \eqref{hs} to the level of Chow motives:   
   
   \begin{proposition}\label{mot} Let notation and assumptions be as in Theorem \ref{proj}. There is an isomorphism of integral Chow motives
     \[  \begin{split}    h(Y)(-3) \oplus \bigoplus_{i=0}^2 h(I_r\Gr(3,n))(-i) \ \xrightarrow{\cong}\  h(Y^\prime) \oplus \bigoplus_{i=1}^{c-1}  h(I_{r+1}\Gr(2,n))(-i)&\\
             \ \ \ \hbox{in}\ \MM^{\ZZ}_{\rm rat}&\ .\\
             \end{split}\]
             (Here $c$ denotes the codimension of $Z^\prime$ in $Y^\prime$.)
     \end{proposition}
     
     \begin{proof} The idea is to express the motive of $\wt{Y}^\prime$ in two different ways:
     
     The blow-up formula expresses $h(\wt{Y}^\prime)$ in terms of $h(Y^\prime)$; this gives the right-hand side of the relation.
     
     Looking at \cite[Section 3.2]{BFM}, one finds that $\wt{Y}^\prime$ is the total space of a projectivization $\PP(E)$ where $E$ is the vector bundle $E:=\OO\oplus \UU^\ast$ on $I_r \Gr(3,n)$
     (in the notation of loc. cit.), and $Y$ is given by a section of $E$. That is, we are in the setting of
      Cayley's trick, and so Theorem \ref{ji} expresses $h(\wt{Y}^\prime)$ in terms of $h(Y)$; this gives the left-hand side of the relation.
      \end{proof}

   \subsection{Hyperplane sections of $I_1 \Gr(3,n)$}
   
   \begin{theorem}\label{main} Let  
      \[ Y:= I_1\Gr(3,n)\cap H \ \ \ \subset\ \PP^{{n\choose 3}-1} \]
  be a smooth hyperplane section (with respect to the Pl\"ucker embedding).   
  Then
      \[ A_i^{hom}(Y)=0\ \ \ \forall\ i\le n-5\ .\]
      Moreover, in case $n\le 10$ or $n=12$ we have
      \[  A_{i}^{hom}(Y)=0\ \ \ \forall\ i\le n-4\ .\]
      \end{theorem}
   
   \begin{proof} A generic hyperplane section $Y$ is attained by the construction of Theorem \ref{proj} with $r=0$, i.e. there is a smooth hyperplane section
     \[Y^\prime:=\Gr(3,n+1)\cap H\ ,\] 
     related to $Y$ via the projection of Theorem \ref{proj}. In this case, Proposition \ref{mot} implies that there is an injection of Chow groups
   \[ A_i^{hom}(Y)\ \hookrightarrow\  A^{hom}_{i+3}(Y^\prime)\oplus \bigoplus A_\ast^{hom}(I_1 \Gr(2,n))\ .\]
   The symplectic Grassmannian $I_1 \Gr(2,n)$ is nothing but a Pl\"ucker hyperplane section of $\Gr(2,n)$, and so Proposition \ref{2n} gives the vanishing
   \[ A_\ast^{hom}(I_1 \Gr(2,n))=0   \ .\]
  The variety $Y^\prime$ is a hyperplane section of $\Gr(3,n+1)$, and so Theorem \ref{3n} gives the vanishing
    \[ A_{i+3}^{hom}(Y^\prime)=0\ \ \ \forall \ i\le n-5\ ,\]
    with the additional vanishing for $i=n-4$ for small $n$. This proves the theorem for generic sections $Y$.
     
  A standard spread argument allows to extend to {\em all\/} smooth hyperplane sections: Let $\YY\to B$ denote the universal family of all smooth hyperplane sections of $\Gr(k,n)$, and let $B^\circ\subset B$ denote the Zariski open subset parametrizing smooth $Y$ verifying the set-up of Theorem \ref{proj}.
Doing the Bloch--Srinivas argument \cite{BS} (cf. also \cite{moi}), the above implies that for each $b\in B^\circ$ one has
   a decomposition of the diagonal
     \begin{equation}\label{decom} \Delta_{Y_b} = \gamma_b+\delta_b \ \ \ \hbox{in}\ A^{\dim Y_b}(Y_b\times Y_b)\ \end{equation}
     where $\gamma_b$ is completely decomposed (i.e. $\gamma_b\in A^\ast(Y_b)\otimes A^\ast(Y_b)$) and $\delta_b$ is supported on
     $Y_b\times W_b$ with $\codim \, W_b=n-2$ (and $\codim \, W_b=n-1$  for small $n$).      
     Using the Hilbert schemes argument of \cite[Proposition 3.7]{V0} (cf. also \cite[Proposition A.1]{LNP} for the precise form used here), the 
     $\gamma_b, \delta_b, W_b$ exist relatively, i.e.
     one can find a cycle $\gamma\in (p_1)^\ast A^\ast(\YY)\cdot (p_2)^\ast A^\ast(\YY)$, a subvariety $ \WW\subset \YY$ of codimension $n-2$, and a cycle $\delta$ supported on $\YY\times_{B^\circ} \WW$ such that
        \[ \Delta_\YY\vert_b= \gamma\vert_b + \delta\vert_b\ \ \ \hbox{in}\ A^{\dim Y_b}(Y_b\times Y_b)\ \ \ \forall\ b\in B^\circ\ .\ \]
     Let $\bar{\gamma}, \bar{\delta} \in A^{\dim Y_b}(\YY\times_B \YY)$ be cycles that restrict to $\gamma$ resp. $\delta$. The spread lemma \cite[Lemma 3.2]{Vo} implies that
      \[ \Delta_\YY\vert_b= \bar{\gamma}\vert_b + \bar{\delta}\vert_b\ \ \ \hbox{in}\ A^{\dim Y_b}(Y_b\times Y_b)\ \ \ \forall\ b\in B\ .\ \]     
   Given any $b_1\in B\setminus B^\circ$, using the moving lemma, one can find representatives for $\bar{\gamma}$ and $\bar{\delta}$ in general position with respect to the fiber $Y_{b_1}\times Y_{b_1}$.   
Restricting to the fiber, this implies that the diagonal of $Y_{b_1}$ has a decomposition as in \eqref{decom}.
 Letting the decomposition \eqref{decom} act on Chow groups, this shows that
     \[  A_i^{hom}(Y_b)=0\ \ \ \forall\ i\le n-5\ ,\ \ \ \forall\ b\in B\ \]
     (with the additional vanishing for $i=n-4$ for small $n$).
      \end{proof}

   \subsection{Hyperplane sections of $I_2 \Gr(3,n)$}
   
   \begin{theorem}\label{main2} Assume $n$ is even, and let
      \[ Y:= I_2 \Gr(3,n)\cap H \ \ \ \subset\ \PP^{{n\choose 3}-1} \]
  be a smooth hyperplane section (with respect to the Pl\"ucker embedding).   
  Then
      \[ A_i^{hom}(Y)=0\ \ \ \forall\ i\le n-7\ .\]
      Moreover, in case $n\le 10$ we have
      \[  A_{i}^{hom}(Y)=0\ \ \ \forall\ i\le n-6\ .\]
     \end{theorem} 
     
     \begin{proof} For $n$ even, a generic hyperplane section $Y$ is attained by the construction of Theorem \ref{proj} with $r=1$, i.e. there is a smooth hyperplane section
       \[Y^\prime:=I_1 \Gr(3,n+1)\cap H\ ,\] 
       related to $Y$ via the projection of Theorem \ref{proj}. In this case, Proposition \ref{mot} implies that there is an injection of Chow groups
   \[ A_i^{hom}(Y)\ \hookrightarrow\  A^{hom}_{i+3}(Y^\prime)\oplus \bigoplus A_\ast^{hom}(I_2 \Gr(2,n))\ .\]
   The bisymplectic Grassmannian $I_2 \Gr(2,n)$ is nothing but an intersection $\Gr(2,n)\cap H_1\cap H_2$ (where the $H_j$ are Pl\"ucker hyperplanes), and so Proposition \ref{2n} gives the vanishing
   \[ A_\ast^{hom}(I_2 \Gr(2,n))=0   \ .\]
  The variety $Y^\prime$ is a hyperplane section of $I_1 \Gr(3,n+1)$, and so Theorem \ref{main} gives the vanishing
    \[ A_{i+3}^{hom}(Y^\prime)=0\ \ \ \forall \ i\le n-7\ ,\]
    with the additional vanishing for $i=n-6$ for small $n$. This proves the theorem for generic sections $Y$.     
    
    The extension to {\em all\/} smooth hyperplane sections $Y$ is done just as in the proof of Theorem \ref{main}.
    \end{proof}

   \section{Some consequences}
   
   \begin{corollary}\label{ghc} 
   
   \noindent
   (\rom1) Let $Y$ be as in Theorem \ref{main} and $n\le 10$ or $n=12$, or as in Theorem \ref{main2} and $n\le 10$. Then $H^{\dim Y}(Y,\QQ)$ is supported on a subvariety of codimension $n-3$.
   
   \noindent
   (\rom2)
   Let $Y$ be as in Theorem \ref{main2} and $n\le 10$. Then $H^{\dim Y}(Y,\QQ)$ is supported on a subvariety of codimension $n-5$.
   \end{corollary} 
   
   \begin{proof} This follows in standard fashion from the Bloch--Srinivas argument \cite{BS}. 
   Let us treat (\rom1) (the argument for (\rom2) is the same). The vanishing 
     \[ A_i^{hom}(Y)=0\ \ \ \forall\ i\le n-4\ \]
   (Theorem \ref{main}) is equivalent to the decomposition
     \[  \Delta_Y = \gamma+\delta\ \ \ \hbox{in}\ A^{\dim Y}(Y\times Y)\ ,\]
     where $\gamma$ is a completely decomposed cycle (i.e. $\gamma\in A^\ast(Y)\otimes A^\ast(Y)$), and $\delta$ has support on $Y\times W$ with $W\subset Y$ of codimension $n-3$
     (to see this equivalence, one can look for instance at \cite[Theorem 1.7]{moi}).
     Let $H^{\dim Y}_{tr}(Y,\QQ)$ denote the transcendental cohomology (i.e. the complement of the algebraic part under the cup product pairing). The cycle $\gamma$ does not act on
     $H^{\dim Y}_{tr}(Y,\QQ)$. The action of $\delta$ on $H^{\dim Y}_{tr}(Y,\QQ)$ factors over $W$, and so
     \[  H^{\dim Y}_{tr}(Y,\QQ)\ \ \subset\ H^{\dim Y}_{W}(Y,\QQ)\ .\]
     Since the algebraic part of $H^{\dim Y}(Y,\QQ)$ is (by definition) supported in codimension $\dim Y/2$, this settles the corollary.
             \end{proof}

\begin{corollary}\label{cor2} Let  
\[ Y:=  I_1 \Gr(3,n)\cap H \ \ \ \subset\ \PP^{{n\choose 3}-1}\]
  be a smooth hyperplane section (with respect to the Pl\"ucker embedding). 
  
 \noindent
 (\rom1) If $n\le 8$, then $Y$ has finite-dimensional motive (in the sense of \cite{Kim}).
 
 \noindent
 (\rom2) If $n\le 9$, then $Y$ has trivial Griffiths groups (and so Voevodsky's smash conjecture \cite{Voe} is true for $Y$, i.e. numerical equivalence and smash-equivalence coincide on $Y$).
 
 \noindent
 (\rom3) If $n\le 10$, the Hodge conjecture is true for $Y$.
\end{corollary}

\begin{proof} This is similar to the argument of Corollary \ref{ghc}.

\noindent
(\rom1) The vanishing
   \[ A_i^{hom}(Y)=0\ \ \ \forall\ i\le n-4\ \]
 (Theorem \ref{main}) is equivalent to the decomposition of the diagonal
   \[   \Delta_Y = \gamma+\delta\ \ \ \hbox{in}\ A^{\dim Y}(Y\times Y)\ ,\]
     where $\gamma$ is a completely decomposed cycle, and $\delta$ has support on $Y\times W$ with $W\subset Y$ of codimension $n-3$
     (cf. \cite{BS} or \cite{moi}). The dimension of $Y$ is $3n-13$, and so (looking at the action of the diagonal) one finds that
     \[ A^\ast_{AJ}(Y)=0\ \]
     as long as $n\le 8$. This implies Kimura finite-dimensionality of $Y$ \cite[Theorem 4]{43}.
     
     \noindent
     (\rom2) The vanishing of Theorem \ref{main} implies that
       \[ \hbox{Niveau}(A_\ast(Y))\le 2 \]
       (in the sense of \cite{moi}), i.e. the motive of $Y$ factors over a surface. Since surfaces have trivial Griffiths groups, the conclusion follows.
       
    \noindent
    (\rom3) The vanishing of Theorem \ref{main} implies that
       \[ \hbox{Niveau}(A_\ast(Y))\le 3 \]
       (in the sense of \cite{moi}), i.e. the motive of $Y$ factors over a threefold. Since threefolds verify the Hodge conjecture, the conclusion follows.
    \end{proof}

  We leave it to the zealous reader to formulate and prove a version of Corollary \ref{cor2} for bisymplectic Grassmannians.
        
 \vskip1cm
\begin{nonumberingt} Thanks to Kai and Len for enjoying Kuifje movies. Thanks to the referee for constructive comments that helped to improve the presentation.
\end{nonumberingt}

%

\end{document}